\documentclass[a4paper,12pt]{amsart}
\usepackage[
hmarginratio={1:1},     
vmarginratio={1:1},     
textwidth=15.8cm,        
textheight=21cm,			
heightrounded,          
]{geometry}        

\usepackage[cp1250]{inputenc}

\usepackage{euler}
\usepackage{graphicx}
\usepackage{lpic}
\usepackage{longtable}
\usepackage{url}
\usepackage{import}
\usepackage{amsfonts, amstext, amsmath, amsthm, amscd, amssymb}
\usepackage[sc]{mathpazo}
\usepackage{enumitem, textcomp, setspace}
\usepackage{caption}
\usepackage{booktabs}

\numberwithin{equation}{section}

\theoremstyle{plain}
\newtheorem{theorem}{Theorem}[section]
\newtheorem{proposition}[theorem]{Proposition}
\newtheorem{lemma}[theorem]{Lemma}

\theoremstyle{definition}

\usepackage[]{hyperref}

\usepackage[all]{hypcap}

\definecolor{newblue}{rgb}{0.27, 0.32, 0.86}
\definecolor{newred}{rgb}{0.86, 0.32, 0.27}
\hypersetup{
	pagebackref=true,
	colorlinks=true,       
	linkcolor=newred,          
	citecolor=newblue,        
	filecolor=magenta,      
	urlcolor=newblue           
}

\usepackage[]{appendix}
\usepackage{wrapfig}

\providecommand{\subjclass}[1]{\textbf{\textit{2020 Mathematics Subject Classification:}} #1}

\graphicspath{ {/} } 

\title{Unity of Jones polynomials in the unit circle and the plane}
\author{Micha\l \;Jab\l onowski}
\address{Institute of Mathematics, Faculty of Mathematics, Physics and Informatics,\newline University of Gda\'nsk, 80-308 Gda\'nsk, Poland}

\email{michal.jablonowski@ug.edu.pl}

\date{\today}

\begin{document}


\maketitle

\begin{abstract}
	In this note, we study solutions of the equation $J_K(t)=1$ for the Jones polynomial of knots and links. For the family $K_n$ of double-twist knots, we show that every root of unity (except $-1$) satisfies $J_{K_n}(\zeta)=1$ for some $n$. Consequently, the set of solutions to $J_{K_n}(t)=1$ arising from this family is dense in the unit circle. We further show that there exists a family of links for which the zeros of $J_L(t)-1$ are dense in the complex plane, adapting the density mechanism of Jin--Zhang--Dong--Tay for Jones polynomial zeros.
\end{abstract}

\section{Introduction}

It is known that the zeros of Jones polynomials $J(t)$ of certain links (and the odd-parameter pretzel subfamily of knots) are dense in the whole complex plane \cite{JZDT10}, and the zeros of Jones polynomials of certain knots are dense in the unit circle \cite{ChaKof05, Mro22}. In this paper, we prove that the zeros of polynomials $J(t)-1$ are dense in the whole complex plane (similarly to \cite{JZDT10}), and that the zeros of polynomials $J(t)-1$ arising from a family of non-trivial knots are dense in the unit circle. 

It is a well-known open problem if any non-trivial knot has $J(t)\equiv1$. The unit circle density of solutions for $J(t)=1$ for non-trivial knots is a result of a computation of zeros of $J(t)-1$ for prime knots up to $16$ crossings in a minimal diagram, visualized in Figure \ref{JWmid_rootsUPto16} (where the "density" of points in the unit circle arcs is visible when the drawing is enlarged). The range of roots in this figure includes $\text{Re}(t)\in[-1,1]$, $\text{Im}(t)\in[0,2]$. Points corresponding to alternating knots are shown in red, and those corresponding to non-alternating knots are shown in blue (under the red points). These points were computed using SageMath \cite{SageMath}.

\begin{figure}[h!t]
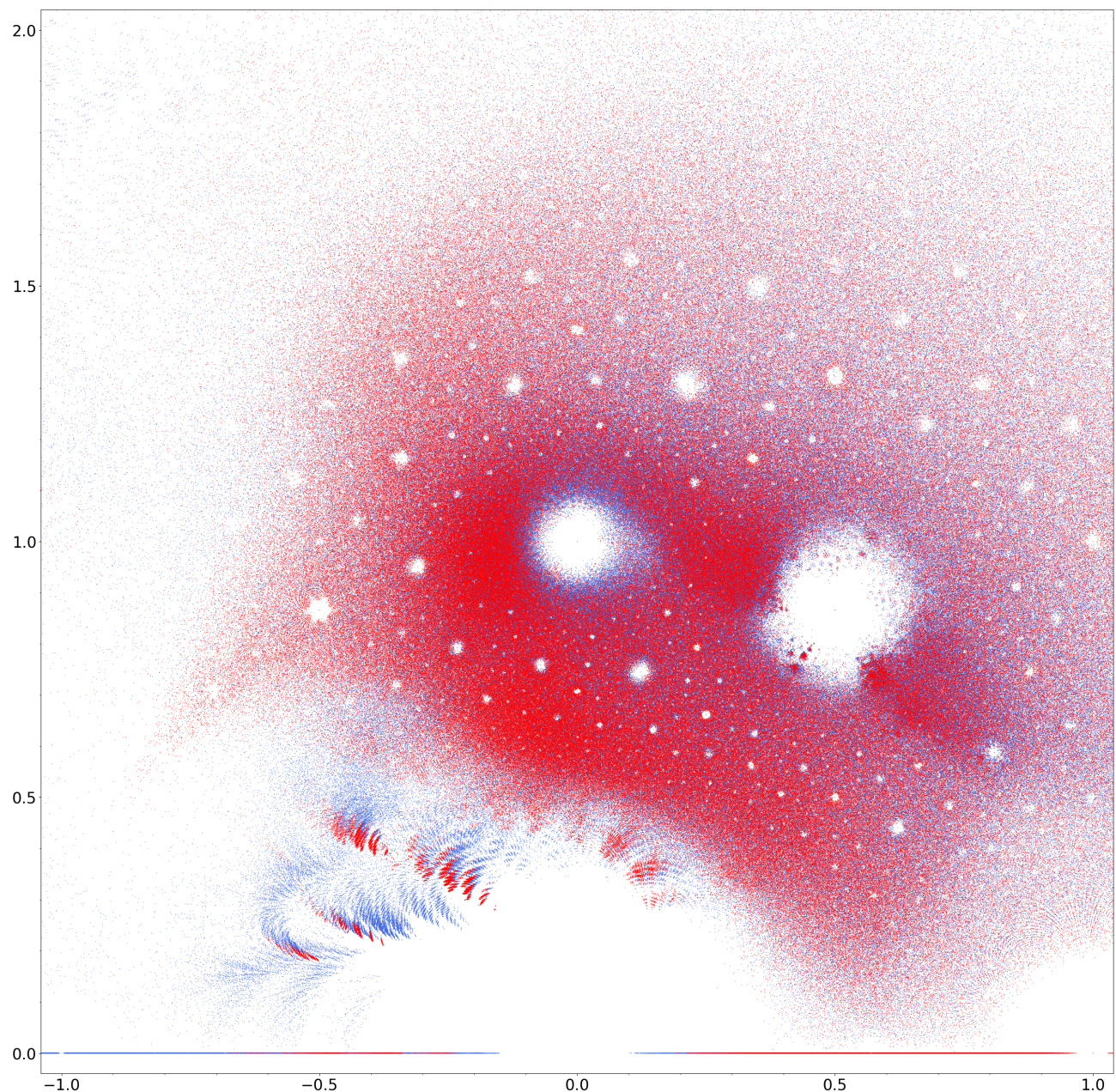

	\begin{center}
		\begin{lpic}[]{JWmid_rootsUPto16altN(17cm)}
		\end{lpic}
		\caption{Solutions to $J(t)=1$ for prime knots up to $16$ crossings.\label{JWmid_rootsUPto16}}
	\end{center}
\end{figure}

\section{Density in the unit circle}

Let $K_n=C(2n,3)$ be the double--twist knot with Conway notation $C(2n,3)$, $n\ge 1$, with one block of $3$ half-twists and one block of $2n$ half-twist crossings. The Rolfsen names for the cases $n=1,2,3$ are $C(2,3)=5_2$, $C(4,3)=7_3$, and $C(6,3)=9_3$.

\pagebreak

\begin{theorem}
	Let $J_n(t)$ denote the Jones polynomial of $K_n$ normalized so that $J_n(1)=1$.
	
	Then
	$
	\left\{t\in S^1:\exists n\ge1,\ J_n(t)=1\right\}
	$
	is dense in the unit circle $S^1$. More precisely,
	\[
	\mu_\infty\setminus\{-1\}
	\subset
	\{\,t\in S^1:\exists n\ge1,\ J_n(t)=1\,\},
	\]
	where $\mu_\infty$ denotes the set of all roots of unity.
\end{theorem}

\begin{proof}
	The knot $K_n$ is alternating and nontrivial. For $n\ge 1$, the corresponding rational fraction is $(6n+1)/3$.
	
	The Jones polynomial admits the exact closed form (see \cite{LawRos25} for generalizations)
	\[
	J_n(t)
	=
	(t^{-3n-3})\cdot\,
	\frac{t^{2n}(1+t^2+t^4)+t^3-t^2-1}{1+t}.
	\]
	
	Let $\zeta$ be a root of unity with $\zeta\neq-1$.
	We solve
	\[
	J_n(\zeta)=1 .
	\]
	
	Multiplying the defining equation by $t^{3n+3}(1+t)$ gives the equivalent relation
	\[
	t^{2n}(1+t^2+t^4)+t^3-t^2-1
	=
	t^{3n+3}(1+t).
	\]
	
	Substituting $t=\zeta$ and writing $x=\zeta^n$, this becomes
	\[
	x^2(1+\zeta^2+\zeta^4)+\zeta^3-\zeta^2-1
	=
	x^3\zeta^3(1+\zeta).
	\]
	
	Rearranging yields the cubic equation
	\[
	\zeta^3(1+\zeta)x^3-(1+\zeta^2+\zeta^4)x^2+(1+\zeta^2-\zeta^3)=0.
	\]
	
	This polynomial factors as
	\[
	(x-1)
	\Big(
	\zeta^3(1+\zeta)x^2
	+
	(\zeta^3-1-\zeta^2)(x+1)
	\Big)
	=0 .
	\]
	
	One solution is
	$
	x=1.
	$
	
	Since $x=\zeta^n$, this condition is equivalent to
	\[
	\zeta^n=1.
	\]
	
	If $N=\operatorname{ord}(\zeta)$, then $n=N$ satisfies this condition.
	Substituting $n=N$ into the defining equation yields
	\[
	J_N(\zeta)=1.
	\]
	
	Thus every root of unity $\zeta\neq-1$ satisfies $J_n(\zeta)=1$ for some $n$.
	
	It remains to consider $\zeta=-1$.
	Because the formula has a removable singularity at $t=-1$, define
	\[
	N_n(t)=t^{2n}(1+t^2+t^4)+t^3-t^2-1.
	\]
	
	Then
	\[
	J_n(-1)
	=
	(-1)^{-3n-3}
	\lim_{t\to-1}\frac{N_n(t)}{1+t}
	=
	(-1)^{-3n-3}N_n'(-1)
	=
	(-1)^n(6n+1),
	\]
	where the second equality follows from l'H\^opital's rule.
	Hence
	\[
	J_n(-1)\neq1
	\quad
	\text{for all } n\ge1.
	\]
	
	Therefore
	\[
	\mu_\infty\setminus\{-1\}
	\subset
	\{\,t\in S^1:\exists n,\ J_n(t)=1\,\}.
	\]
	
	Since the set of roots of unity is dense in $S^1$, the set 
	$
	\{\,t\in S^1:\exists n,\ J_n(t)=1\,\}
	$ 
	
	is also dense in $S^1$.
\end{proof}

\section{Density in the complex plane}

\begin{proposition}\label{prop:BKW3}
	There exists a family of links $\{L_n\}_{n\ge 1}$ such that the zeros of
	\[
	J_{L_n}(t)-1
	\]
	are dense in the complex plane $\mathbb{C}$.
\end{proposition}

The Jin–Zhang–Dong–Tay \cite{JZDT10} mechanism adapts from $J$ to $J-1$, with extra steps.

\begin{lemma}[Beraha--Kahane--Weiss \cite{BKW78}]\label{lem:BKW}
	Let
	\[
	F_n(z)=\sum_{j=1}^k \alpha_j(z)\lambda_j(z)^n
	\]
	be a sequence of analytic functions in a domain $D\subset\mathbb{C}$,
	where $\alpha_j(z)$ and $\lambda_j(z)$ are analytic and the $\lambda_j$
	are not identically proportional.  
	
	If $z_0\in D$ satisfies
	\[
	|\lambda_i(z_0)|=|\lambda_j(z_0)|>\max_{k\neq i,j}|\lambda_k(z_0)|,
	\]
	then $z_0$ is a limit point of zeros of $\{F_n\}$.
\end{lemma}

\begin{proof}[Proof of Proposition \ref{prop:BKW3}]
	Jin--Zhang--Dong--Tay \cite{JZDT10} construct families of links
	$L_n(T)$ (and the odd-parameter pretzel subfamily of knots) obtained by forming rings of identical tangles $T$.
	For such families, the Jones polynomial admits a representation
	(up to multiplication by a monomial $\pm t^{k/2}$) of the form
	\[
	J_{L_n(T)}(t)
	=
	a(t)\Lambda_1(t)^n+b(t)\Lambda_2(t)^n .
	\]
	
	Define
	\[
	JW_n(t)=J_{L_n(T)}(t)-1
	=
	a(t)\Lambda_1(t)^n+b(t)\Lambda_2(t)^n+(-1)\cdot 1^n.
	\]
	
	Let $t_0\in\mathbb{C}$ and $\varepsilon>0$.
	By Lemma~3.4 of \cite{JZDT10}, there exists a tangle $T$ and a point $t_*$ satisfying
	\[
	|t_*-t_0|<\varepsilon/2
	\]
	such that
	\[
	|\Lambda_1(t_*)|=|\Lambda_2(t_*)|.
	\]
	Moreover, the constructions $I_s^+$ and $I_s^-$ in \cite{JZDT10}
	allow the eigenvalues $\Lambda_1,\Lambda_2$ to be scaled
	multiplicatively, as follows.
	
	Let $H$ be a signed plane graph corresponding to a $2$--tangle $T$.

	Let $I_s$ be the graph with two vertices joined by $s$ parallel edges.
	Denote by $I_s^+$ and $I_s^-$ the signed graphs obtained by assigning all edges
	positive and negative signs respectively.
	
	For $I_s^+$ the coefficients appearing in the Tutte polynomial construction gives the relation
	\[
	|1-(-t)^s|
	=
	|1+(t+t^{-1}+1)(-t)^s|.
	\]
	
	Thus, the parameter $s$ enters only through the power $t^s$.
	
	Then the eigenvalues of the transfer matrix corresponding to $T_s^{\pm}$ satisfy
	\[
	\Lambda_i^{(s)}(t) = c_{\pm}(t)^s \, \Lambda_i(t), \qquad i=1,2,
	\]
	for some nonzero function $c_{\pm}(t)$ depending only on the sign choice.
	Consequently,
	\[
	|\Lambda_1^{(s)}(t)| = |c_{\pm}(t)|^s |\Lambda_1(t)|,
	\qquad
	|\Lambda_2^{(s)}(t)| = |c_{\pm}(t)|^s |\Lambda_2(t)|.
	\]
	
	Therefore replacing an edge of $H$ by $I_s^{\pm}$ multiplies each
	eigenvalue $\Lambda_i(t)$ by a common factor $c_{\pm}(t)^s$.
	
	The same computation for $I_s^-$ yields the factor $t^{-s}$.
	Hence, in both cases, the eigenvalues are scaled multiplicatively while preserving their ratio.
	
	Hence the equality
	$|\Lambda_1(t_*)|=|\Lambda_2(t_*)|$
	can be arranged so that
	
	$|\Lambda_1(t_*)|=|\Lambda_2(t_*)|>1$.
	
	Thus $JW_n$ is of the form required in Lemma~\ref{lem:BKW}, with three eigenvalues
	$\Lambda_1,\Lambda_2,$ and $1$. Since
	\[
	|\Lambda_1(t_*)|=|\Lambda_2(t_*)|>|1|,
	\]
	Lemma~\ref{lem:BKW} implies that $t_*$ is a limit point of zeros of $\{JW_n\}$.
	
	Hence for some $n$ there exists a zero $t$ of $JW_n$ such that
	\[
	|t-t_*|<\varepsilon/2.
	\]
	Therefore
	$
	|t-t_0|<\varepsilon.
	$
	
	Since $t_0$ and $\varepsilon$ were arbitrary, the zeros of
	$J_{L_n}(t)-1$ are dense in $\mathbb{C}$.
\end{proof}


\begin{thebibliography}{99}
	
	
	\bibitem{BKW78} S. Beraha, J. Kahane, and N.J. Weiss, Limits of zeros of recursively defined families of polynomials, \emph{Studies in Foundations and Combinatorics, Advances in Mathematics Supplementary Studies}, 1 (1978) 212--232.
	
	\bibitem{ChaKof05} A. Champanerkar, I. Kofman, On the Mahler measure of Jones polynomials under twisting, \emph{Algebraic \& Geometric Topology}, 5(1) (2005) 1--22.
	
	\bibitem{SageMath} Developers, The~Sage, \emph{{S}agemath, the {S}age {M}athematics {S}oftware {S}ystem ({V}ersion 10.6)}, (2026),\\ \url{https://www.sagemath.org}
	
	\bibitem{JZDT10} X. Jin, F. Zhang, F. Dong, E.G. Tay, Zeros of the Jones polynomial are dense in the complex plane, \emph{The Electronic Journal of Combinatorics}, (2010) \#R94.
	
	\bibitem{LawRos25} R. Lawrence and O. Rosenstein, Jones rational coincidences, \emph{Journal of Knot Theory and Its Ramifications}, Vol. 34.03 (2025) 2340015.
	
	\bibitem{Mro22} M. Mroczkowski, Infinitely Many Roots of Unity Are Zeros of Some Jones Polynomials, \emph{Geometriae Dedicata}, 216, no. 4 (2022) 43.
	
	
	
\end{thebibliography}
\end{document}